\documentclass[12pt,reqno]{amsart}
\usepackage{amssymb}
\usepackage{graphicx}
\usepackage{amsmath}
\usepackage{a4wide}

\newtheorem{theorem}{Theorem}

\newtheorem{cor}[theorem]{Corollary}
\newtheorem{lemma}[theorem]{Lemma}
\newtheorem{proposition}[theorem]{Proposition}

\newtheorem{conj}[theorem]{Conjecture}

\theoremstyle{remark}

\makeatletter
\@namedef{subjclassname@2010}{
  \textup{2010} Mathematics Subject Classification}
\makeatother

\def\la{\lambda}
\def\l{\ell}
\def\Z{\mathbb Z}
\def\N{\mathbb N}
\def\R{\mathbb R}
\def\bee{\begin{enumerate}}
\def\ene{\end{enumerate}}
\def\bei{\begin{itemize}}\def\eni{\end{itemize}}
\def\beq{\begin{equation}}\def\enq{\end{equation}}
\def\beqs{\begin{equation*}}\def\enqs{\end{equation*}}
\def\C{\mathbb C}

\def\N{\mathbb N}
\def\R{\mathbb R}
\def\Q{\mathbb Q}\def\T{\mathbb T}
\def\Z{\mathbb Z}
\def\al{\alpha}

\def\phi{\varphi}

\def\r{\mathbf r}\def\z{\mathbf z}\def\0{\mathbf 0}\def\s{\mathbf s}
\def\c{\mathbf c}\def\j{\mathbf j}
\def\GL{\operatorname{GL}}

\def\max{\operatorname{max}}
\def\dim{\operatorname{dim}}

\begin{document}
\title[Closed sets of Mahler measures]
{Closed sets of Mahler measures}
\author{ Chris Smyth}
\address{School of Mathematics and Maxwell Institute for Mathematical Sciences\\
University of Edinburgh\\
Edinburgh EH9 3FD\\
Scotland, U.K.}
\email{C.Smyth@ed.ac.uk}
\date{22 August 2017}

\begin{abstract}
Given a $k$-variable Laurent polynomial $F$,  any $\l\times k$ integer matrix $A$ naturally defines an $\l$-variable Laurent polynomial $F_A.$ I prove that for fixed $F$ the set $\mathcal M(F)$  of all the logarithmic   
Mahler measures $m(F_A)$  of $F_A$ for all $A$ is a closed subset of the real line.  Moreover, the matrices $A$ can be assumed to be of a special form, which I call {\it Saturated Hermite Normal Form}. Furthermore, if $F$ has integer coefficients and $\mathcal M(F)$ contains $0,$ then $0$ is an isolated point of this set.

I also show that, for a given bound $B>0$,  the set ${\mathcal M}_B$ of all Mahler measures of integer polynomials in any number of variables and having length (sum of the moduli of its coefficients) at most $B$ is closed. Again, $0$ is an isolated point of ${\mathcal M}_B$.

These results constitute evidence consistent with a conjecture of Boyd from 1980 to the effect that the union $\mathcal L$ of all sets ${\mathcal M}_B$ for $B>0$ is closed, with $0$ an isolated point of $\mathcal L$.
\end{abstract}

\maketitle

\section{Introduction}\label{S-introduction}

The Mahler measure $M(f)$ of a polynomial $f(z)=a\prod_i(z-\al_i)\in\C[z]$ is defined by 
$M(f):=|a|\prod_i \max(1,|\al_i|)$. It arose first in a paper of D.H. Lehmer \cite{Le},
as a way of estimating the growth rate of integer sequences defined by a  linear recurrence. (Lehmer was using such sequences to generate primes.) Later, Kurt Mahler \cite{Ma64} used it to bound the discriminant of a polynomial from above, and the minimum spacing of the roots of a polynomial from below. As noted for instance in \cite{Ma60}, $m(f):=\log M(f)$ has the integral representation via Jensen's Theorem as 
\[
m(f)=\int_0^1 \log|f(e^{2\pi it})|dt.
\]
This formula has the advantage that it immediately suggests a generalisation of Mahler measure to polynomials in several variables. So, following  Mahler \cite{Ma62},
let $k\ge 1$, $\z_k=(z_1,\dots,z_k)$ and $F(\z_k)$ be a nonzero Laurent polynomial with complex coefficients. Then its (logarithmic) Mahler measure $m(F)$ is defined as 
\begin{equation}\label{E-1}
m(F)=\int_0^1\cdots\int_0^1\log|F(e^{2\pi it_1},\dots,e^{2\pi it_k})|\,dt_1\cdots dt_k
\end{equation}
 Its (classical) Mahler measure is then $M(F)=\exp(m(F))$; for our purposes here it is slightly more convenient to work with $m(F)$ rather than $M(F).$

Next,  given $\l\ge 0$ and an $\l\times k$ matrix $A=(a_{ij})\in\Z^{\l\times k}$, define, following \cite{Sc}, the $k$-tuple $\z_\l^A$ by
\beq\label{E-0}
\z_\l^A= (z_1,\dots,z_\ell)^A=(z_1^{a_{11}}\!\cdots z_\ell^{a_{\ell 1}},\ldots,z_1^{a_{1k}}\!\cdots z_\ell^{a_{\ell k}})
\enq
(which is $(1,1,\dots,1)\in\Z^k$ when $\l=0$) and $F_A(\z_\l)=F(\z_\l^A)$, a polynomial in $\ell$ variables $z_1,\ldots,z_\ell$. Then $m(F_A)$ is defined by \eqref{E-1} with $F=F_A$ and $k=\ell$. Denote  by $\mathcal P(F)$ the set 

\[
\mathcal P(F):=\{F_A\, : \, A\in\Z^{\l\times k}, \l\ge 0 \},
\]

 and by $\mathcal M(F)$ the set
\[
\mathcal M(F):=\{m(F_A)\, : \, F_A\in\mathcal P(F), F_A\ne 0 \}.
\]

In particular, taking $\l=1$, we write $A$ as the row vector $\r=(r_1,\dots,r_k)\in\Z^k$, and so
 $F_{\r}(z)=F(z^{r_1},\dots,z^{r_k})$. Denote by ${\mathcal M}_1(F)$  the set $\{m(F_{\r})\mid \r\in\Z^k\}$, with $m(F_{\r})$ defined by \eqref{E-1} with $F=F_\r$ and $k=1$.

 Our first result is as follows.

\begin{theorem} \label{T-1} Let $F$ be a nonzero Laurent polynomial with complex coefficients. Then the set $\mathcal M(F)$ is a closed subset of $\R$. Moreover, it is the closure $\overline{{\mathcal M}_1(F)}$ of ${\mathcal M}_1(F)$ in $\R$.
\end{theorem}

So $\overline{{\mathcal M}_1(F)}={\mathcal M}(F)=\overline{{\mathcal M}(F)}.$
Thus not only are all the $m(F_A)$  in 
$\overline{{\mathcal M}_1(F)}$, but also they are the {\it only} elements of $\overline{{\mathcal M}_1(F)}$.

In Proposition \ref{P-lbd} a lower bound for ${\mathcal M}(F)$ is given, which implies that ${\mathcal M}(F)\subseteq [0,\infty)$ when $F$ has integer coefficients. In this case  
all polynomials in ${\mathcal P}(F)$ also have integer coefficients, and we can state our next result.

\begin{theorem} \label{T-1.5} Suppose that $k\ge 1$, that  the nonzero  Laurent polynomial $F(\z_k)$ has integer coefficients and that $0\in\mathcal M(F)$. Then $0$ is an isolated point of $\mathcal M(F)$.
\end{theorem}

The case where $F(\z_k)$ is a linear form was proved in 1977 by Lawton \cite{L2}.

When $F$ has integer coefficients and $\mathcal M(F)$ consists of more than just the single element $0$, Theorems  \ref{T-1} and \ref{T-1.5} tell us that this set has a smallest positive element. I call this the {\it Lehmer element of} $\mathcal M(F)$, and denote it by $\l_{\min}(F)$. It is discussed in Section \ref{S-Leh}.

 These results are consistent with a far-sighted conjecture of Boyd \cite{BoSpec} to the effect that the set $\mathcal{L}$ of Mahler measures of {\it all} polynomials with integer coefficients in any number of variables is a closed subset of $\R.$  While the theorems are some way from the full conjecture, they do, I think, represent the first substantive results in the direction of such a proof.  As Boyd pointed out, the truth of his conjecture implies easily that $0$ is an isolated point of  $\mathcal{L}.$ It is clear that Boyd's set $\mathcal{L}$
is a countable union of sets ${\mathcal M}(F)$. In fact $\mathcal{L}$ can be written
as a countable {\it nested} union of a sequence of such sets -- see Proposition \ref{P-nest}.

 For $k$-variable polynomials $F$ having integer coefficients, the arithmetic nature of $m(F)$ is an interesting one. For $k=1$, $M(F)$ is algebraic, so $m(F)$, if not $0$, is transcendental. For a few $F$ with $k\ge 2$, there are explicit formulae for $m(F)$, one example being (for $k=3$, from \cite{Sm2})
\[
m(1+z_1^{-1}+z_2+(1+z_1+z_2)z_3)=\frac{14}{3\pi^2}\zeta(3).
\]
Other deeper such formulae involve $L$-functions of various kinds, evaluated at specific integers. See the survey of Bertin and Lal\'\i n \cite{BL}, and also Papanikolas {\it et al} \cite{PRS} for a more recent example. There may even be some connection between Mahler measure of integer polynomials and the Feynman integrals of mathematical physics -- see for instance Samart \cite{Sa} and Vanhove \cite{Va}.

Theorems \ref{T-1} and \ref{T-1.5} can be used to prove another result in the direction of Boyd's conjecture.

\begin{theorem} \label{T-1.75} Let $B>0$ be given.  The set ${\mathcal M}_B$ of all Mahler measures of integer polynomials in any number of variables and having length (sum of the moduli of its coefficients) at most $B$ is closed. Furthermore, $0$ is an isolated point of the set.
\end{theorem}

 Actually, the fact that $0$ is an isolated point of ${\mathcal M}_B$ already follows from
results of Mignotte \cite{Mign78} (see also  Stewart\cite{Stew78b}, \cite[Section 5.4]{Sm3}), who gave the lower bound $2^{1/(2B)}$ for the Mahler measure of integer noncyclotomic one-variable polynomials of  length at most $B$. On applying the fact, from Theorem \ref{T-1}, that ${\mathcal M}_1(F)$ is dense in ${\mathcal M}(F)$, we see that Mignotte's lower bound is valid for all positive elements of ${\mathcal M}_B$.

In another paper \cite[Theorem 2]{DS}, Dobrowolski and I have recently proved a theorem similar to Theorem \ref{T-1.75}, but for the set of Mahler measures of integer polynomials that are sums of a bounded number of monomials.

Boyd's conjecture is a substantial generalisation of a question asked by Lehmer \cite{Le} in 1933 as to whether, in the set $\{m(F)\}$ of Mahler measures of all nonzero one-variable polynomials $F$ with integer coefficients, the point $0$ is isolated. As is well known, the smallest known such $m(F)>0$ is $m(L(z))=\log(1.1762808)=0.1623576\cdots$,
where
\beq\label{E-Leh}
L(z)=z^{10}+z^9-z^7-z^6-z^5-z^4-z^3+z+1,
\enq
first discovered by Lehmer himself \cite{Le}. Lehmer's question went unstudied for many years -- see the survey \cite{Sm3} -- but has in recent decades become one of the central unsolved problems in algebraic number theory. The truth of Boyd's Conjecture would answer Lehmer's question affirmatively.

 In Section \ref{S-Q} I state a new conjecture, which implies Boyd's conjecture. Whether this conjecture turns out to be any more tractable than his will be interesting to see! In that section, I also state another conjecture which, if provable, would answer a question Boyd posed at the end of \cite{BoSpec}.

We can in fact severely restrict the matrices $A$ that are needed to produce all the different elements of $\mathcal M(F)$  in Theorem \ref{T-1}. To do this, we need to make the following definition.
We say that a matrix $H\in\Z^{\l\times k}$ is in {\it Saturated Hermite Normal Form}
if $H$ is in Hermite normal form and the intersection of the $\R$-vector space spanned by the rows of $H$ with $\Z^k$ is equal to the integer lattice spanned by the rows of $H$.  The lattice spanned by the rows of $H$ is then a socalled {\it saturated} lattice (see e.g., \cite[p. 13]{ES}).
For discussion of this form of an integer matrix, see Section \ref{S-mat} below.

 \begin{theorem} \label{T-2} We have
\[
{\mathcal M}(F)=\{m(F_H)\, : \, H\in\cup_{\l=0}^k\Z^{\l\times k}, \text{  $H$ of rank $\l$ in Saturated Hermite Normal Form} \}.
\]
\end{theorem}
 In particular, this result shows that elements $m(F_A)$ make no additional contribution to ${\mathcal M}(F)$ when $A$ has more rows than columns. Further, by Proposition \ref{P-kk} below, the only $m(F_A)$ with $A$ a square matrix that makes a contribution is the $k\times k$ identity matrix $I_k$, giving $m(F_{I_k})=m(F).$

Using this result, we can describe ${\mathcal M}(F)$ explicitly for small $k$, for instance for  $k=1,2,3$. For $k=1$ there are two types: $m(F(1))$ (corresponding to the empty matrix), and $m(F(z))$, corresponding to $I_1$. For $k=2,3$, see Tables \ref{Ta-1} and \ref{Ta-2}.

\begin{table}[htbp]
\begin{tabular}{|c|c|c|c|}
  \hline
 rank($A$)  &    Mahler measure    & matrix $A$  & range of exponents\\
  \hline
 $0$ & $m(F(1,1))$ & empty matrix & \\
$1$ & $m(F(1,z_1))$ & $(0\,1)$ & \\
$1$ & $m(F(z_1^n,z_1^p))$ & $(n\, p)$ & $n\in\N$, $p\in\Z$, $\gcd(n,p)=1$\\
$2$ & $m(F(z_1,z_2))$ & $I_2$ & \\
\hline
\end{tabular}
\caption{${\mathcal M}(F)$ for $k=2$.  }
 \label{Ta-1}
 \end{table}

\begin{table}[htbp]
\begin{tabular}{|c|c|c|c|}
  \hline
 rank($A$)  &    Mahler measure    & matrix $A$  & ranges of exponents\\
  \hline
$0$ & $m(F(1,1,1))$ & empty matrix & \\
$1$ & $m(F(1,1,z_1))$ & $(0\,0\,1)$ & \\
$1$ & $m(F(1,z_1^n,z_1^p))$ & $(0\, n\, p)$ & $n\in\N$, $p\in\Z$, $\gcd(n,p)=1$\\
$1$ & $m(F(z_1^n,z_1^p,z_1^q))$ & $(n\, p \, q)$ & $n\in\N$, $p,q\in\Z$, $\gcd(n,p,q)=1$\\
$2$ & $m(F(1,z_1,z_2))$ & $\left( \begin{array}{ccc}
0 & 1 & 0 \\
0 & 0 & 1 \end{array} \right)$ & \\
$2$ & $m(F(z_1^n,z_1^p,z_2))$ & $\left( \begin{array}{ccc}
n & p & 0 \\
0 & 0 & 1 \end{array} \right)$ & $n\in\N$, $p\in\Z$, $\gcd(n,p)=1$\\
$2$ & $m(F(z_1^n,z_1^{n'}z_2^{p'},z_1^{n''}z_2^{p''}))$ & $\left( \begin{array}{ccc}
n & n' & n'' \\
0 & p' & p'' \end{array} \right)$ & $n,p\in\N$, $n',n'',p',p''\in\Z$, $0\le n'<p'$,\\
& & &  $\gcd(n,p)=\gcd(n',p')=\gcd(p',p'')=1$\\
 & &  &  $\gcd(n,n'-rp',n''-rp'')=1$ for $0\le r<n$\\
$3$ & $m(F(z_1,z_2,z_3))$ & $I_3$ & \\
\hline
\end{tabular}
\caption{${\mathcal M}(F)$ for $k=3$.  }
 \label{Ta-2}
 \end{table}

From Theorem \ref{T-2} it is clear that the elements of the multiset $\{m(F_A)\, : \, A\in\Z^{\l\times k} \}$ given in Theorem \ref{T-1} are not all different. But, furthermore, 
I am not even claiming that all the measures given in Theorem \ref{T-2} are distinct. Indeed,
 Proposition \ref{P-3}, which follows, shows that ${\mathcal M}(F)$ can be $\{0\}$.

\section{Preliminaries}\label{S-prelims}

We now present some results needed for the proofs of Theorems \ref{T-1} and \ref{T-1.5}.

\begin{proposition}[{{Boyd \cite[Theorem 1]{Bo1} -- see also  {Schinzel \cite[Section 3.4, Cor.17, p. 260]{Sc}} and Smyth \cite[Cor. 1]{Sm1}}}]\label{P-3} Suppose that $F\in\Z[\z_k]$ for some $k\in\N$.  Then  $m(F)=0$ if and only if $F$ belongs to $\mathcal P(S)$ for some polynomial $S$ of the form $S(\z_k)=\pm z_1C_2(z_2)C_3(z_3)\cdots C_k(z_k)$ for some $k$, where $C_2,\dots,C_k$ are cyclotomic polynomials.
\end{proposition}

\begin{lemma}  \label{L-minus1} If  $B\in\Z^{\l'\times\l}$ and $A\in\Z^{\l\times k}$ then
$\left(\z_\l^{B}\right)^A=\z_\l^{(BA)}$. Further, if $G=F_A\in\mathcal P(F)$ for some polynomial $F$, then $G_B=F_{BA}.$
\end{lemma}

\begin{proof} The first result comes from \cite[Section 3.4, Cor. 1, p. 223]{Sc}.
For the second result, we have
\[
G_B(\z_{\l'})=(F_A)_B(\z_{\l'})=F_A(\z_{\l'}^B)=F((\z_{\l'}^B)^A)=F(\z_{\l'}^{BA})=F_{BA}(\z_{\l'}),
\]
as claimed.
\end{proof}

\begin{proposition}\label{P-1} For a polynomial $G(\z_\l)$ and nonsingular $V\in\Z^{\l\times \l}$  we have $m(G(\z_\l^V))=m(G(\z_\l))$ and $m(G_V)=m(G)$. Further, for a polynomial $F(\z_k)$ and any $\l\times k$ integer matrix $A$ we have $m(F_{VA})=m(F_A).$ 
\end{proposition}
\begin{proof} For $m(G(\z_\l^V))=m(G(\z_\l))$, see \cite[Lemma 7]{Sm2}. (See also Schinzel \cite[Section 3.4, Cor. 8, p. 226]{Sc} for the case $\det(V)=\pm 1$.) Then $m(G_V)=m(G)$ follows straight from the definition of $G_V.$
Next, we have
\[
m(F_{VA}(\z_\l))=m(F(\z_\l^{VA}))=m(F((\z_\l)^V)^A)=m(F_A(\z_\l^V))=m(F_A(\z_\l)).
\]

\end{proof}

\begin{lemma}  \label{L-00}  If $G\in\mathcal P(F)$ then $\mathcal P(G)\subseteq\mathcal P(F)$ and $\mathcal M(G)\subseteq\mathcal M(F)$.
\end{lemma}

\begin{proof} Suppose that $F=F(\z_k)$ and $G=G(\z_\l)\in\mathcal P(F)$. Then $G=F_A$ for some $A\in\Z^{\l\times k}$, 
and for  any  $B\in\Z^{\l'\times \l}$ with $0\le\l'\le\l$ we have $G_B=F_{BA}$
by Lemma \ref{L-minus1}. Note that
$BA\in\Z^{\l'\times k}$ with $0\le\l'\le\l\le k$. This proves the first assertion, from which the second assertion follows immediately.
\end{proof}

\begin{lemma}  \label{L-000} For any two multivariable Laurent polynomials $F$ and $G$ we have that $\mathcal P(FG)\subseteq\mathcal P(F)\mathcal P(G)$ and $\mathcal M(FG)\subseteq\mathcal M(F)+\mathcal M(G)$.
\end{lemma}

\begin{proof} For $F,G$ polynomials in $\z_k$ and $A\in\Z^{\l\times k}$ and some $\l$ with $0\le\l\le k$ we have
\[
(FG)_A=F_AG_A\in\mathcal P(F)\mathcal P(G),
\]
and hence $m((FG)_A)=m(F_A)+m(G_A)\in\mathcal M(F)+\mathcal M(G)$.
\end{proof}

This immediately implies the following.

\begin{cor} \label{C-cyccc} If $\mathcal M(G)=\{0\}$ (see Proposition \ref{P-3}), then $\mathcal M(FG)=\mathcal M(F)$.
\end{cor}

Next, given $\l\ge 2$ and $\r=(r_1,\dots,r_\l)\in\Z^\l$, define, following Boyd \cite{Bo1,BoSpec}
\[
q(\r)=\min_{\0\ne\s\in\Z^\l}\{\max_{i=1}^\l|s_i|\, : \, \r\cdot\s=0\}.
\]
Here $\s=(s_1,\dots,s_\l)$. The function $q$ measures, in some sense, how different in magnitude the $r_i$ are.

\begin{lemma}[{{Boyd \cite[p. 118]{Bo1}}}]\label{L-1} Let $n\in\N$ and $\r_n=(1,n,n^2,\dots,n^{\ell-1})$. Then $q(\r_n)=n$ (and so goes to $\infty$ as $n\to\infty)$.
\end{lemma}

The next result was first conjectured by Boyd \cite{Bo1}, who also proved in \cite{BoSpec} some partial results in direction of his conjecture, including essentially the result for $k=2$.

\begin{proposition}[{{Lawton\cite{L}}}]\label{P-2} Let $F(z_1,\dots,z_k)$ be a Laurent polynomial with complex coefficients, and suppose that $\r^{(1)}, \r^{(2)},\dots,\r^{(n)},\dots$ is a sequence of vectors in $\Z^k$ with $q(\r^{(n)})\to\infty$ as $n\to\infty$. Then
\[
\lim_{n\to\infty}m\left(F_{\r^{(n)}}\right)=m(F).
\]
\end{proposition}

Earlier Boyd  had proved this result when $F$ does not vanish on the $k$-torus $\T^k$, this being a special case of \cite[Lemma 1]{Bo1}, which states that for a continuous function $f:\T^k\to\C$
\[
\lim_{n\to\infty}\int_\T f\left(z^{\r^{(n)}}\right)dz=\int_{\T^k}f(\z_k)d\z_k
\]
for the same sequence of vectors $(\r^{(n)})$.

\begin{proposition} [{{Smyth \cite[Cor. 2]{Sm1}}}]\label{P-lbd} For a Laurent polynomial $F(\z_k)=\sum_{\j\in J}c(\j)\z_k^{\j}\in\C[z_1,\dots,z_k]$, where $J\subseteq\Z^k$, let the polytope $\mathcal C(F)\in\R^k$  be the convex hull of those $\j\in J$ with $c(\j)\ne 0$. Then
\[
\max_{\j \text{ an extreme point of }\mathcal C(F)}\log|c(\j)|\le m(F)\le\log\left(\sum_{\j\in J}|c(\j)|\right)
\]
In particular, $m(F)\ge 0$ when $F$ has integer coefficients.
\end{proposition}

Here $J$ is a set of column vectors, so that $\z_k^{\j}$, defined by \eqref{E-0}, is a monomial.
The polytope $\mathcal C(F)$ is called the {\it exponent polytope of $F.$} \cite[p. 460]{BoSpec}. Let $\dim(F)$ denote its dimension, the {\it dimension of $F$}, which is clearly at most $k$.

\begin{proposition}\label{P-nest}  Boyd's set $\mathcal{L}$ can be written as a union 
$\cup_{n=1}^\infty {\mathcal M}(F^{(n)})$, where
\[
F^{(n)}(\z_{2n})=z_1+z_3+\cdots+z_{2n-1}-(z_2+z_4+\cdots+z_{2n}).
\]
Furthermore,
\[
{\mathcal M}(F^{(1)})\subseteq {\mathcal M}(F^{(2)})\subseteq{\mathcal M}(F^{(3)})\subseteq\dots \subseteq{\mathcal M}(F^{(n)})\subseteq\cdots\quad .
\]
\end{proposition}

\begin{proof} 
For a given $F(\z_k)=\sum_{\j\in J}c(\j)\z_k^{\j}\in\Z[z_1,\dots,z_k]$, we know that ${\mathcal M}((z_1-1)F(\z_k))={\mathcal M}(F(\z_k))$, by Corollary \ref{C-cyccc}. So, replacing $F$ by $(z_1-1)F$, if necessary, we can assume that $\sum_{\j\in J}c(\j)=F(1,\dots,1)=0$. Choose
\[
n:= \sum_{\j\in J \text{ with } c(\j)>0}c(\j).
\]
Then, for each $\j$ with $c(\j)>0$ replace $c(\j)$ of the $z_{2i-1}$'s by $\z_k^\j$, and for 
each $\j$ with $c(\j)<0$ replace $(-c(\j))$ of the $z_{2i}$'s by $\z_k^\j$. This gives us the polynomial $F$ in the form $F^{(n)}_A$, where $A$ is a matrix where for each $\j\in J$ the matrix $A$ has $|c(\j)|$ of its columns equal to $\j$. Hence $m(F)\in{\mathcal M}(F^{(n)})$
for this value of $n$. 

To show that the sequence of sets $\left({\mathcal M}(F^{(n)})\right)$ are nested, it is enough to observe that 
\[
F^{(n-1)}(\z_{2n-2})=F^{(n)}(z_1,z_2,\dots,z_{2n-3},z_{2n-2},z_{2n-2},z_{2n-2}),
\]
so that
 $F^{(n-1)}$  is of the form $F^{(n)}_A$ for some (easily written down) matrix $A$, and hence that $F^{(n-1)}\in {\mathcal P}(F^{(n)})$. Applying Lemma \ref{L-00}, we see that   ${\mathcal M}(F^{(n-1)})\subseteq{\mathcal M}(F^{(n)})$, as claimed.

\end{proof}

\section{The Saturated Hermite Normal Form (SHNF) of an integer matrix}\label{S-mat} The canonical form for integer matrices which we now describe is a variant of the classical (row-echelon) Hermite Normal Form.
Recall that a nonzero integer $\l\times k$ matrix $A=(a_{ij})$ is in {\it Hermite Normal Form} (HNF) if it has the following properties:
\bee
\item For some integer $r$ with $0\le r<k$ the leftmost $r$ columns of $A$ are zero;
\item For some integer $s$ with $0\le s<\l$ the bottom $s$ rows of $A$ are zero;
\item For $1\le i\le \ell-s$ there are integers $j_i$ satisfying  
$$
r+1 = j_1<j_2<\cdots <j_{\ell-s}\le k
$$
 such that $a_{i,j_i}\in\N$,
$a_{ij'}=0$ for $j'<j_i$ and $0\le a_{i'j_i}<a_{i,j_i}$ for $i'< i$.
\ene
Note that $A$ has rank $\ell-s$, and that $0\le \ell-s\le k-r$.
 %(The version of HNF  described here is the row-echelon one.)

For any nonzero integer $\l\times k$ matrix $A$ there is some $U\in\GL_\ell(\Z)$ such that $UA$ is in Hermite normal form -- see \cite[Ch. II, Section 6]{N}. (Multiplying by $U$ corresponds to applying a succession of row operations of the following kinds to $A$:
 \bei
\item (`Swap') Swap two rows;
\item (`Sign-change') Cange the sign of one row;
\item (`Add') Add an integer multiple of one row to a different row.
\eni
This matrix $H=UA$ is then called the {\it Hermite Normal form of $A$}.

We now give a  characterisation of matrices in {\it Saturated Hermite Normal Form} (SHNF) alternative to the definition given in the Introduction. To streamline our
 discussion, we ignore any all-zero rows in our matrices.

\begin{proposition} \label{P-Y}An nonzero integer $\l\times k$ matrix $A$ with rows $\underline{a}_1,$ $\underline{a}_2,$ $\dots,\underline{a}_\ell$ is in SHNF if it is in HNF, and, additionally,
for $i=1,2,\dots,\ell$ and every choice of integers $u_{i+1},\dots,u_\ell$ the greatest common divisor ($\gcd$) of all the components of 
$\underline{a}_i+\sum_{j=i+1}^\ell u_j\underline{a}_j$ is $1$.
\end{proposition}
\begin{proof} First assume that $H$ is in SHNF as defined in the introduction. Then, for any real numbers $\la_1,\dots,\la_\l$, if $\sum_{j=1}^\l \la_j \underline{a}_j\in\Z^k$, it follows that the $\la_j$ are all integers.  But if $\underline{a}_i+\sum_{j=i+1}^\ell u_j\underline{a}_j$
has components with $\gcd=g$ for some integers $u_{i+1},\dots,u_{k}$ then 
$\frac{1}{g}\underline{a}_i+\sum_{j=i+1}^\ell \frac{u_j}{g}\underline{a}_j\in\Z^k$. Therefore, by our assumption, $g$ must be $1$.

Conversely, assume $A$ is in HNF and that for each $i=1,2,\dots,\l$ and for each choice of integers $u_i$ that the components of  
$\underline{a}_i+\sum_{j=i+1}^\ell u_j\underline{a}_j$ have $\gcd=1$. For any set of real $\la_j$'s, assume that the  sum $\sum_{j=1}^\l \la_j\underline{a}_j$ has integer components. From the fact that $A$ is in row-echelon form, we see successively that $\la_1\in\Q, \la_2\in\Q,\dots,\la_\l\in\Q$. Write $\la_j=n_j/g$, where the $n_j$ are integers with $\gcd=1$, and $g$ (the least common denominator of the $\la_j$'s) is a positive integer. Now suppose that $g>1$, with say $p$ a prime dividing $g$. Then the  sum $\sum_{j=1}^\l n_j\underline{a}_j$ has all components divisible by $p$.
Suppose that $h$ is the smallest index for which $p\nmid n_h$, and that $vn_h=1\pmod{p}$. Then the sum 
$\underline{a}_h+vn_{h+1}\underline{a}_{h+1}+\dots+vn_{\ell}\underline{a}_\ell$ also has all its components divisible by $p$. But this contradicts our $\gcd=1$ assumption above. Hence $g=1$, so that the $\la_j$ are all integers.
\end{proof}

If we allow ourselves to apply the following  operation to $A$, additional to `Swap', `Sign-change' and `Add' above:
 \bei
\item (`Scale') If a row has all entries divisible by $g>1$, then divide that row by $g$,
\eni
then we can reduce $A$ to a matrix in SHNF. We call this matrix the {\it Saturated Hermite Normal Form of $A$}.

\

{\bf SHNF algorithm.\quad} First put $A\in\Z^{\l\times k}$ into HNF -- call it $A$ again. Ignore any zero rows at the bottom of $A$, so that we can assume that $A$ has rank $\l\le k$. We now start an $\l$-step process. The first step is to divide row $\l$ by the $\gcd$ of its entries. Now add a suitable multiple of (the new) row $\l$ to row $\l-1$ so that the $\gcd$ of the entries of (the new) row $\l-1$ is as large as possible. Divide this row by this $\gcd$.
Now add  suitable multiples of  rows $\l$ and row $\l-1$ to row $\l-2$ so that the $\gcd$ of the entries of (the new) row $\l-2$ is as large as possible. Divide this row by this $\gcd$.
Continue in this way. For the $\l$-th step, add  suitable multiples of  rows $\l,\l-1,\dots,2$ to row $1$ so that the $\gcd$ of the entries of (the new) row $1$ is as large as possible. Divide this row by this $\gcd$. Finally, restore the resulting matrix to HNF by suitable row operations.

\

Incidentally, it should be clear from the row-echelon structure of $A$ that the number of choices for sums of integer multiples of the lower rows
 to be added to the current row to give a $\gcd$ greater than $1$ can be readily bounded.

\begin{proposition}\label{P-55} For a matrix $A\in\Z^{\l\times k}$, the above algorithm does indeed find its SHNF,  $H\in\Z^{\l\times k}$. Furthermore,  we have $A=VH$ for some nonsingular $V\in\Z^{\l\times \l}$.
\end{proposition}
\begin{proof} Suppose that the algorithm does not find the SHNF $H$ of $A$. Then, by Proposition \ref{P-Y}, some row $\underline{a}_j$ of $H$ has the property that, for some integrs $u_j$,  the sum 
$\underline{a}_i+\sum_{j=i+1}^\ell u_j\underline{a}_j$  has all its components divisible by some $g>1$. But this readily leads to the conclusion that, in applying the algorithm to row $i$,
the sums of multiples of the lowers rows that were added did not give the largest possible $\gcd$ of the components, which it should have. Thus the algorithm works.

On applying the algorithm to $A$, each operation corresponds to either left multiplication of $A$ by an element of $\GL_\l(\Z)$ or left multiplication by $D^{-1}$, where $D$ is an $\l\times\l$ diagonal matrix, all entries except one being $1$, and the other entry being some integer $g>1$. Hence indeed $A=VH$ for some nonsingular $V\in\Z^{\l\times \l}$.
\end{proof}

{\bf Example.\quad} Consider the matrix $$A=\left( \begin{array}{cccc}
1 & 1 & 4 & 0 \\
0 & 2 & 3 & 3\\
0 & 0 & 5 & 1 \end{array} \right),$$
which is already in HNF. However, it is not in SHNF. To find its SHNF, add row 3 to row 2, and divide the modified row 2 by $2$. Finally, subtract the twice-modified row 2 from row 1, to obtain $$\left( \begin{array}{cccc}
1 & 0 & 0 & -2 \\
0 & 1 & 4 & 2\\
0 & 0 & 5 & 1 \end{array} \right),$$ the SHNF of $A$.

\

The following result is an easy exercise in applying the SHNF algorithm.

\begin{proposition}\label{P-kk}
If $A\in\Z^{k\times k}$ is nonsingular then its SHNF is the $k\times k$ identity matrix.
\end{proposition}

The next result is needed for the proof of Theorem \ref{T-2}.

\begin{proposition}\label{P-77} Suppose that $A\in\Z^{{\l'}\times k}$ has rank $\l\le k$ and that its Saturated Hermite Normal Form is 
$H'=\left( \begin{array}{c} H \\
{\mathbf 0} \end{array} \right)$,
 where $H\in\Z^{\l\times k}$ and ${\mathbf 0}$ is the $(\l'-\l)\times k$ zero matrix. Then $m(F_{A})=m(F_H)$.
\end{proposition}
\begin{proof} We have that $A=VH'$ for some nonsingular $V\in\Z^{k\times k}$. Hence $m(F_{A})=m(F_{H'})$ by Proposition \ref{P-1}. Furthermore, $F_{H'}=F_{H}$, by definition.
\end{proof}

\section{Proof of  Theorem \ref{T-1}}\label{S-proof}
\begin{proof}
For the proof, we first show that every $m(F_A)$ lies in $\overline{{\mathcal M}_1(F)}$. Then we show for any  $\r^{(1)}, \r^{(2)},\dots,\r^{(n)},\dots$ in $\Z^k$ with $m\left(F_{\r^{(n)}}\right)$ converging, that its limit is of the form $m(F_A)$ for some $A$.

So, take any $A\in\Z^{\l\times k}$,  and let $\r^{(n)}$ be the sequence $(1,n,n^2,\dots,n^{\l-1})$ from Lemma \ref{L-1}. Because $q(\r^{(n)})\to\infty$ as $n\to\infty$ we can apply Proposition \ref{P-2} to $F_A$ to obtain
\[
\lim_{n\to\infty}m\left(F_A(\r^{(n)})\right)=m(F_A).
\]
 Now for $\r=\r^{(n)}$ we have $F_A(\r)=F_{\r^A}$, so that $m(F_A(\r))=m(F_{\r^A})$. 
Hence, for the sequence $\r^A=(\r^{(n)})^A\in\Z^k$ we have
\[
\lim_{n\to\infty}m(F_{\r^A})=m(F_A).
\]
Hence ${\mathcal M}(F)=\{m(F_A)\, : \, A\in\Z^{\l\times k} \}\subseteq\overline{{\mathcal M}_1(F)}$.

To prove that these are the only limit points of ${\mathcal M}_1(F)$, we take any sequence 
\[
\r^{(1)}, \dots,\r^{(n)},\dots\in \Z^k
\]
 for which $m(F_{\r^{(n)}})$ converges. We separate the proof into three cases, doing the trivial case $k=1$ first and then, for $k\ge 2$,  separating the cases where the sequence $q(\r^{(n)})$ is either unbounded or bounded.

 {\bf Case 1: $\mathbf{k=1}$}. Here $F=F(z_1)$ and our sequence is $\{m(F(z^{r_1}))\}$, for some sequence of nonzero integers $r_1$. But, applying  Proposition \ref{P-1} with $n=1$ and $A=(r_1)$, we have that the sequence $\{m(F(z^{r_1}))\}$ is constant, each term being $m(F(z))$. 

\

 {\bf Case 2: $\mathbf{k\ge 2}$ and $\mathbf{q(\r^{(n)})}$  unbounded}. Then there is a subsequence of the $\r^{(n)}$'s
for which $\lim_{n\to\infty}q(\r^{(n)})$ tends to infinity on that subsequence. Thus, by replacing the sequence of the $\r^{(n)}$'s by that subsequence, we can assume that, as $n\to\infty$ both that $m(F_{\r^{(n)}})$ converges and  that $q(\r^{(n)})\to\infty$. Then we can apply Proposition \ref{P-2} to conclude that $\lim_{n\to\infty}m(F_{\r^{(n)}})=m(F)$.

\

 {\bf Case 3: $\mathbf{k\ge 2}$ and  $\mathbf{q(\r^{(n)})}$  bounded}. Our proof is by induction. From Case 1, we already know that the result is true for $k=1$. We now assume $k\ge 2$ and that the result is true for all Laurent polynomials $F$ in fewer than $k$ variables.

Take a convergent sequence of real numbers $m(F_{\r})$ for $\r=\r^{(n)}\quad(n=1,2,3,\dots)$ such that the integer sequence $(q(\r))$ is bounded. Then there are only finitely many possibilities for the nonzero vectors $\s\in\Z^k$ in the definition of $q$ such that $\s\cdot\r=0$. Hence, by the Pigeonhole Principle,  we can find an infinite subsequence of integers $n$ for which the corresponding sequence of vectors $\s$ is constant. On replacing our original sequence $n=1,2,3,\dots$ by this subsequence, we can assume that {\it all} $\r$ satisfy $\r\cdot \s=0$.

Next, take a $(k-1)\times k$ integer matrix $U$ whose rows are a basis  of the sublattice $L_{\s}:=\{\r\in\Z^k\mid \r\cdot \s=0\}$ of $\Z^k$. Then each $\r\in L_{\s}$ can be written as $\c U$ for some $\c\in\Z^{k-1}$. Then  writing $G(\z_{k-1}):=F_U(\z_{k-1})$,
a Laurent polynomial in at most $k-1$ variables, we have from Lemma \ref{L-minus1} that
$G_\c=F_{\c U}=F_\r.$ Hence, applying the induction hypothesis to $G$, or Case 2 if $k-1\ge 2$ and the sequence $\{q(\c)\}$ is unbounded, we see that the sequence 
$\{m(F_{\r}(z))\}=\{m(G_\c(z))\}$ has a limit of the form $m(G_B)$ for some $B\in\Z^{\l\times(k-1)}$ and some $\ell\le k-1$.

Next, we note that, by Lemma \ref{L-minus1} again, $G_B=F_A$, where $A=BU\in\Z^{\l\times k}$. Hence $\overline{{\mathcal M}_1(F)}\subseteq 
\{m(F_A)\, : \, A\in\Z^{\l\times k} \}={\mathcal M}(F)$, and so 
$\overline{{\mathcal M}_1(F)} = {\mathcal M}(F)$, as claimed.
\end{proof}

\section{Proof of  Theorem \ref{T-1.5}}\label{S-proof1.5}

\begin{proof} Suppose that $0\in{\mathcal M}(F)$, but that it is not isolated. Then, because this set is the closure of the set of measures $m(F_\r)$ of polynomials $F_\r$ for $\r\in\Z^k$, we can take a sequence of such polynomials $F_{\r^{(n)}}$ such that none of the 
$m(F_{\r^{(n)}})$ are $0$, but $\lim_{n\to\infty}m(F_{\r^{(n)}})=0$. However, by Proposition 
\ref{P-2}, this limit is $m(F)$, which is therefore $0$.

As in the proof of Theorem \ref{T-1}, we now separate three cases.

\

{\bf Case 1: $\mathbf{k=1}$}. Here $F=F(z_1)$ and our sequence is $\{m(F(z^{r_1}))\}$, for some sequence of nonzero integers $r_1$. But, as in Case 1 of the proof of Theorem \ref{T-1}, we have that
 the sequence $\{m(F(z^{r_1}))\}$ is constant, each term being $m(F)$. Hence $m(F)=0$, and so all terms of the converging sequence are $0$, contrary to our assumption.

\

{\bf Case 2: $\mathbf{k\ge 2}$ and $\mathbf{q(\r^{(n)})}$  unbounded}.
 Then, as in the proof of Case 2 of Theorem \ref{T-1}, there is a subsequence of the $\r^{(n)}$'s
for which $\lim_{n\to\infty}q(\r^{(n)})$ tends to infinity on that subsequence. Thus, by replacing the sequence of the $\r^{(n)}$'s by that subsequence, we can assume that, as $n\to\infty$ both that $m(F_{\r^{(n)}})\to 0$  and  that $q(\r^{(n)})\to\infty$.

 Hence, by Proposition \ref{P-3}, $F$ is of the form $\pm z$ times a product of  cyclotomic polynomials $C(z)$, where each occurence of the variable $z$ is replaced by a (possibly different for each occurence) monomial in $z_1,\dots,z_k$. Hence each $F_\r$ is of the form $\pm z$ times a product of  cyclotomic polynomials $C(z)$, where each occurence of the variable $z$ is replaced by a (possibly different for each occurence) power of $z$, assumed to be nonzero. So $m(F_{\r^{(n)}})=0,$ contradicting the fact that these values are all assumed to be nonzero.

\

{\bf Case 3: $\mathbf{k\ge 2}$ and $\mathbf{q(\r^{(n)})}$  bounded}. Here, we follow quite closely the induction argument in Case 3 of the proof of Theorem \ref{T-1}. Thus the result is true for $k=1$ by Case 1, so we assume that $k\ge 2$ and that the result is true for all $F$ in fewer than $k$ variables. Following that argument, we get that our sequence $\{m(F_{\r^{(n)}})\}$ has limit $m(F_{A'})$, where $A'\in\Z^{\ell\times k}$ for some $\ell\le k-1$. Thus
$m(F_{A'})=0$, and so, again  by Proposition \ref{P-3}, $F$ is of the form $\pm z$ times a product of  cyclotomic polynomials $C(z)$, where each occurence of the variable $z$ is replaced by a (possibly different for each occurence) monomial in $z_1,\dots,z_k$. From the definition of $F_{A'}$, we then see that $F$ itself has the same property. So, as in Case 2, we conclude that  $m(F_{\r^{(n)}})=0$ for all $n$, giving the same contradiction again.

\end{proof}

\section{Proof of  Theorem \ref{T-1.75}}\label{S-proof1.75}

\begin{proof} Consider  all signed partitions $\c=(c_1,c_2,\dots,c_t)$ of all positive integers $b\le B$.  So the $c_i$ are nonzero integers with $|c_1|,\dots,|c_t|$ nondecreasing and $\sum_{i=1}^t|c_i|=b$. For each such partition define the linear form
\[
F_{\c}(\z_t)=c_1z_1+\cdots+c_tz_t.
\]
As $F_\c$ is of length (the sum of the moduli of its coefficients) $b$,  all $(F_\c)_A\in\mathcal P(F_\c)$ are of length at most $b$ (there could be cancellation). Then every polynomial $F$ of length at most $B$ belongs to $\mathcal P(F_\c)$ for some such $\c$, and so the set ${\mathcal M}_B$ is the union
of all such sets ${\mathcal M}(F_\c)$. Since this is a finite union, with all the component sets being closed, and with $0$ being an isolated point of each set, it follows that  ${\mathcal M}_B$ inherits these two properties.

\end{proof}

\section{Proof of  Theorem \ref{T-2}}\label{S-proof2}

\begin{proof}
Obviously
\[
\{m(F_A)\, : \, A\in\cup_{\l=0}^k\Z^{\l\times k}, \text{  $A$ of rank $\l$ in SHNF} \}\subseteq 
 \{m(F_A)\, : \, A\in\Z^{{\l'}\times k}, \l'\ge 0 \} = {\mathcal M}(F).
\]
In the other direction, we have by   Proposition \ref{P-77} that   $m(F_A)=m(F_H)$, where $H'=\left( \begin{matrix}
H \\
{\mathbf 0} \end{matrix} \right)$ is the SHNF of $A$, with
 $H\in\Z^{\l\times k}$ of rank $\l\le \min(\l',k).$ Hence
\[
{\mathcal M}(F)=\{m(F_A)\, : \, A\in\Z^{\l'\times k} \}\subseteq 
\{m(F_H)\, : \, H\in\cup_{\l=0}^k\Z^{\l\times k}, \text{  $H$ of rank $\l$ in SHNF} \}.
\]
\end{proof}

\section{The Lehmer element of ${\mathcal M}(F^{(n)})$}\label{S-Leh}
Recall from the introduction that the Lehmer element $\l_{\min}(F)$ of $\mathcal M(F)$ is its smallest positive element.
Now $F^{(1)}=z_1-z_2=z_2(z_1z_2^{-1}-1)$, so $\mathcal M(F^{(1)})=\{0\}$ by 
Proposition \ref{P-3}, and $\l_{\min}(F^{(1)})$ is not defined.
Next, $z^4-z^3-z^2+1\in\mathcal P(F^{(2)})$, and 
\[
m(z^4-z^3-z^2+1)=m((z^3-z-1)(z-1))=m(z^3-z-1)=\log(1.3247\dots),
\]
 so that $\l_{\min}(F^{(2)})\le m(z^3-z-1)$.  But in fact $\l_{\min}(F^{(2)})= m(z^3-z-1)$, by virtue of a result of Dobrowolski \cite[Proposition 2]{Dob}, who in fact proved that $m(z^3-z-1)$ was the minimal Mahler measure of {\it all} integer noncyclotomic quadrinomials.

Also $z^{11}-z^9-z^8+z^3+z^2-1\in\mathcal P(F^{(3)})$, and 
\[
m(z^{11}-z^9-z^8+z^3+z^2-1)=m(L(z)(z-1))=m(L(z))=\log(1.1762808\dots),
\]
where Lehmer's polynomial $L(z)$ is given by \eqref{E-Leh}.
This shows that  $\l_{\min}(F^{(2)}))\le\log(1.176\dots)$. Because the $\mathcal M(F^{(n)})$'s are nested, it follows that $\l_{\min}(F^{(n)}))\le m(L(z))$ for all $n\ge 3$.

Of course, being closed and clearly bounded, ${\mathcal M}(F)$ has a maximal element as well;
let us call it $\l_{\max}(F)$. For instance, since $m(z^4+z^2-z-1)=\log(1.75487766624669)=2m(z^3-z-1)$, we have $\l_{\max}(F^{(2)})\ge 2\l_{\min}(F^{(2)})$.

\section{Questions and conjectures}\label{S-Q}

\bee 
\item[1.] Which points of ${\mathcal M}(F)$ are true limit points? For such a limit point, $m(F_A)$ say,  there is a sequence $(\r^{(n)})$ for which  $m\left(F_{\r^{(n)}}\right)\to m(F_A)$ with all $m\left(F_{\r^{(n)}}\right)$ distinct. This question has been considered by Boyd \cite[Appendix 2]{BoSpec} for the polynomial $F(z_1,z_2)=1+z_1+z_2$. See also the discussion in \cite[Section 6]{BM}.

\item[2.] Given $k\ge 2$,  does there exist a integer polynomial $F$ in $k$ variables   and of dimension $k$ --- see Section \ref{S-prelims} for the definition --- for which the elements of the multiset 
$\{m(F_H)\, : \, H\in\cup_{\l=0}^k\Z^{\l\times k}, \text{  $H$ of rank $\l$ in SHNF} \}$ 
from Theorem \ref{T-2} are all distinct?

\ene
\

 I make the following conjectures concerning Boyd's set $\mathcal L$ of Mahler measures $m(F)$, where $F$ is an integer polynomial in any number of variables.

 \begin{conj}\label{Conj-1} In any sequence $(m(F_n))_{n\in\N,F_n\in\mathcal L}$ where only finitely many of the $m(F_n)$ belong to any one set $\mathcal M(F)$ ($F\in\mathcal L$) we have that $m(F_n)\to\infty$ as $n\to\infty.$
\end{conj}
\

If this conjecture is true, then any convergent sequence in $\mathcal L$ must have an infinite convergent subsequence in some $\mathcal M(F)$. As $\mathcal M(F)$ is closed, by Theorem \ref{T-1}, the sequence's limit is in $\mathcal M(F)$, and so is certainly in $\mathcal L$. This shows that Conjecture \ref{Conj-1} implies Boyd's conjecture.

 \begin{conj}\label{Conj-2} Suppose that we have a sequence $(m(F_n))_{n\in\N}$ where the $F_n\in\mathcal L$ are irreducible and $\dim(F_n)\to\infty$ as $n\to\infty.$ Then only finitely many of the $m(F_n)$ belong to any one set $\mathcal M(F)$ ($F\in\mathcal L$).
\end{conj}

\

If both conjectures are true, we see that any sequence $\{m(F_n)\}_{n\in\N}$ where the $F_n\in\mathcal L$ are irreducible and $\dim(F_n)\to\infty$ as $n\to\infty$ would have the property that $m(F_n)\to\infty$ as $n\to\infty.$ This would answer affirmatively a question asked by Boyd \cite[p. 461]{BoSpec}.

\section{Acknowledgements.} 
I thank the referee for carefully reading the manuscript, and pointing out some infelicities.
I also thank Val\' erie Flammang for an exchange of emails on the Mahler measure of trinomials,
 which got me thinking about limits of Mahler measures again. 

This paper has its origins in my visit to UBC, Vancouver, B.C., in 1979-80, at the invitation of David Boyd. As you see, its gestation time has been considerable!
 I (very belatedly) thank him in print for the opportunity to work with him, which greatly stimulated my research. More generally, I am very grateful to him for his inspiration, support and friendship over the years. 

 Finally, I thank Hendrik Lenstra for `saturated'.

\end{document}